\documentclass[11pt,a4paper,titling,fleqn]{article} 

\usepackage[paper=a4paper,left=20mm,right=20mm,top=20mm,bottom=20mm]{geometry}

\RequirePackage[numbers]{natbib}

\usepackage{amsthm,amsmath,amsfonts,amssymb}
\usepackage{graphicx,psfrag,epsf}
\usepackage{enumerate}
\usepackage{url}
\usepackage{hyperref} 
\hypersetup{colorlinks=true,allcolors=[rgb]{0,0,1}}
\usepackage{color,xcolor}
\usepackage{dsfont}
\usepackage{multirow}
\usepackage{array}
\usepackage{booktabs}

\allowdisplaybreaks
\selectfont

\theoremstyle{plain}
\newtheorem{theorem}{Theorem}[section]

\newtheorem{corollary}[theorem]{Corollary}

\theoremstyle{definition}

\newtheorem{example}[theorem]{Example}

\theoremstyle{remark}

\newcommand{\E}{\mathbb{E}}
\newcommand{\N}{\mathbb{N}}
\newcommand{\R}{\mathbb{R}}
\newcommand{\PP}{\mathbb{P}}

\newcommand{\XX}{\mathbf{X}}

\newcommand{\ww}{\mathbf{w}}
\newcommand{\xx}{\mathbf{x}}
\newcommand{\zz}{\mathbf{z}}

\newcommand{\eqd}{\stackrel{\mathrm{d}}=}

\newcommand{\Var}{\mathrm{Var}}
\newcommand{\de}{\mathrm{\,d}}
\newcommand{\supp}{\mathop{\mathrm{supp}}}

\definecolor{light-gray}{gray}{0.95}
\definecolor{darkblue}{rgb}{0,0,.5}
\definecolor{foxred}{rgb}{0.7, 0.11, 0.11}
\newcommand{\Mail}[1]{{\color{foxred} #1}}

\begin{document}


\title{\bf A dimension reduction for extreme types \\ of directed dependence}

\author{%
	Sebastian Fuchs\footnote{University of Salzburg, Austria, Email: \Mail{sebastian.fuchs@plus.ac.at}}
	\qquad
	Carsten Limbach\footnote{University of Salzburg, Austria, Email: \Mail{carsten.limbach@plus.ac.at}} 
	}
	\vspace{5mm}

\maketitle

\begin{abstract}
In recent years, a variety of novel measures of dependence have been introduced being capable of characterizing diverse types of directed dependence, 
hence diverse types of how a number of predictor variables $\mathbf{X} = (X_1, \dots, X_p)$, \(p \in \N\), may affect a response variable $Y$. 
This includes perfect dependence of $Y$ on $\XX$ and independence between $\XX$ and $Y$, but also less well-known concepts such as zero-explainability, stochastic comparability and complete separation.
Certain such measures offer a representation in terms of the Markov product $(Y,Y')$, with $Y'$ being a conditionally independent copy of $Y$ given $\XX$. 
This dimension reduction principle allows these measures to be estimated via the powerful nearest neighbor based estimation principle introduced in \cite{chatterjee2021AoS}.
To achieve a deeper insight into the dimension reduction principle, this paper aims at translating the extreme variants of directed dependence, typically formulated in terms of the random vector $(\XX,Y)$, into the Markov product $(Y,Y')$.
\end{abstract}

\noindent%


\section{Introduction}

Quantifying directed dependence constitutes a cornerstone of dependence modeling. 
The term `directed dependence' thereby covers a wide range of different types, each of which requires a specific measure of directed dependence for its evaluation.
We here focus on three such measures that, while capable of quantifying very different forms of directed dependence, share a common characteristic:
\begin{enumerate}
\item 
Azadkia and Chatterjee's \emph{simple measure of conditional dependence} $\xi$ introduced in \cite{chatterjee2021AoS} is given (in its unconditional form), for $Y$ being non-degenerate, by
\begin{align}\label{Eq.Chatterjee}
  \xi(Y|\XX) 
  & := \frac{\int_{\mathbb{R}} \operatorname{Var}(\PP(Y \geq y \mid \mathbf{X})) \, \de \PP^Y(y)}
{\int_{\mathbb{R}} \operatorname{Var}(\mathds{1}_{\{Y \geq y\}}) \, \de \PP^Y(y)}\,.
\end{align}
The coefficient \(\xi\) takes on values in the interval \([0,1]\). 
Moreover, due to \cite{chatterjee2021AoS}
\begin{enumerate}[(i)]
\item $\xi(Y|\XX)=0$ if and only if $Y$ and $\XX$ are independent. 
\item $\xi(Y|\XX)=1$ if and only if $Y$ \emph{perfectly depends} on $\XX$, i.e.~there exists some measurable function $f$ such that $Y = f(\XX)$ almost surely.
\end{enumerate}
In a regression setting, $\xi(Y|\XX)$ determines the extent of functional dependence of $Y$ given the information contained in the predictor variables \(\XX = (X_1,\dots,X_p)\).

\item 
Pearson's \emph{correlation ratio} denoted by $R^2$ and introduced in \cite{pearson1905} is given, for $Y \in L^2$ being non-degenerate, by
\begin{align}\label{Eq.Pearson}
  R^2(Y|\XX) 
  & := \frac{\operatorname{Var}(\mathbb{E}(Y|\mathbf{X}))}{\operatorname{Var}(Y)}\,.
\end{align}
The coefficient $R^2$ takes on values in the interval \([0,1]\).
Moreover, due to \cite{ansari2023DepM}
\begin{enumerate}[(i)]
\item $R^2(Y|\XX)=0$ if and only if $Y$ is not explainable through $\XX$ (\emph{zero-explainability}), 
i.e. \linebreak $\mathbb{E}(Y|\mathbf{X}) = \mathbb{E}(Y)$ almost surely.  
\item $R^2(Y|\XX)=1$ if and only if $Y$ perfectly depends on $\XX$.
\end{enumerate}
$R^2$ 
determines the proportion of variance that is explained by the regression function $r(\xx) = \mathbb{E}(Y|\mathbf{X} = \xx)$; see also \cite{fuchs2023JMVA, shih2024, emura2021}.
A value $R^2(Y|\XX)=0$ indicates no variability in the conditional expectation. 
In a regression model this means that the predictor vector $\XX$ provides no explanation for the variance of the response variable $Y$.

\item 
The \emph{coefficient of separation} $\Lambda$ introduced in \cite{fuchs2025} is given, for $Y$ and at least one coordinate of $\XX$ being non-degenerate, by 
\begin{align}\label{Eq.Fuchs}
  \hspace{-10mm}
	\Lambda(Y|\XX)
  & := \frac{4}{1-\PP(\XX = \XX^\ast)} 
       \left( \; \int_{\R^p \times \R^p} \left( \Psi \big(\PP^{Y|\XX=\xx_1}, \PP^{Y|\XX=\xx_2}\big) - \frac{1}{2} \right)^2 
      \de (\PP^\XX \otimes \PP^\XX) (\xx_1,\xx_2) \right),
\end{align}
where \(\XX^\ast\) denotes an independent copy of \(\XX\).
Here, $\Psi$ denotes the relative effect \cite{birnbaum1957, Brunner2019, mann1947, wilcoxon1945} given, for two random variables $Z_1$ and $Z_2$, by
\begin{align*} 
  \Psi(\PP^{Z_1}, \PP^{Z_2}) 
  & := \int_{\mathbb{R}} \PP(Z_1 < z) + \frac{1}{2} \; \PP(Z_1 = z) \; \de \PP^{Z_2}(z)\,.
\end{align*} 
The relative effect $\Psi(\PP^{Z_1}, \PP^{Z_2})$ determines the stochastic tendency of $Z_2$ to take on greater values than $Z_1$.
It is a well-established statistical tool employed in medicine \cite{schimke2022severe} and the social sciences \cite{seidel2014modeling} for comparing the distributions of two treatment groups.
If $Z_2$ shows no (stochastic) tendency to take on greater or smaller values than $Z_1$, then \(\Psi(\PP^{Z_1}, \PP^{Z_2}) = 1/2\) and $Z_1$ and $Z_2$ are said to be \emph{stochastically comparable} \citep{Brunner2019}. 
If, instead, \(\Psi(\PP^{Z_1}, \PP^{Z_2}) \in \{0,1\}\), then $Z_1$ and $Z_2$ are said to be \emph{completely separated} \citep{Brunner2019}.
The coefficient $\Lambda$ is a proper generalization of $\Psi$ to an arbitrary number of treatment groups and takes on values in the interval \([0,1]\).
Moreover, due to \cite{fuchs2025}
\begin{enumerate}[(i)]
\item $\Lambda(Y|\XX)=0$ if and only if $Y$ is \emph{stochastically comparable} relative to $\XX$, i.e.
\linebreak $ \Psi \big(\PP^{Y|\XX=\xx_1}, \PP^{Y|\XX=\xx_2}\big) = 1/2$
for all almost all $(\xx_1,\xx_2)$ with $\xx_1 \neq \xx_2$.
\item $\Lambda(Y|\XX)=1$ if and only if \(Y\) is \emph{completely separated} relative to \(\XX\), i.e.
\linebreak $\Psi \big(\PP^{Y|\XX=\xx_1}, \PP^{Y|\XX=\xx_2}\big) \in \{0,1\}$ for almost all $(\xx_1,\xx_2)$ with $\xx_1 \neq \xx_2$.
\end{enumerate}
A value $\Lambda(Y|\XX)=0$ indicates that the values of $Y$ show no location effect relative to $\XX$ (cf.~Examples \ref{stochasticcomp} - \ref{Vkick}). 
Instead, a value $\Lambda(Y|\XX)=1$ describes the situation when the supports of the conditional distributions are almost surely pairwise disjoint.
\end{enumerate}
Clearly, independence between $\XX$ and $Y$ implies $\xi(Y|\XX) = R^2(Y|\XX) = \Lambda(Y|\XX) = 0$. 
The reverse implication, however, only applies to $\xi$ as mentioned above.
For counterexamples concerning $R^2$ and $\Lambda$ we refer to Section~\ref{Sec.MinimumValue}.
In contrast, perfect dependence and complete separation are generally not connected as is illustrated in Section~\ref{Sec.MaximumValue}.

The three aforementioned measures of directed dependence share a common characteristic: in fact, they all are functionals of the Markov product resulting in a statistical problem with reduced dimension:
Denote by \(Y'\) a conditional independent copy of \(Y\) given \(\XX\,,\) i.e.
\begin{align}\label{MP}
  (Y'|\XX = \xx ) \eqd (Y|\XX = \xx) 
  \textrm{ for } \PP^\XX\text{-almost all } \xx\in \R^p \text{ and } Y \perp Y' \mid \XX \,,
\end{align}
where \(\eqd\) indicates equality in distribution and \(Y \perp Y' \mid \XX\) denotes conditional independence of \(Y\) and \(Y'\) given \(\XX\). 
Then, according to \cite{ansari2025Cont,fuchs2023JMVA},
\begin{align}\label{MP:DF}
    \PP (Y \leq y, Y' \leq y')
    & = \E \big( \PP (Y \leq y \, | \, \XX) \, \PP (Y' \leq y' \, | \, \XX) \big)
\end{align}
for all $(y,y') \in \R^2$; we refer to \cite{fuchs2023JMVA, limbach2024} for more background on how the transformation affects certain distributions.
Moreover,
\begin{enumerate}
\item \(\xi\) fulfills
\begin{align}\label{Xi:Rep}
  \xi(Y|\XX) 
  & = a \int_{\R} \PP(Y < y, Y' < y) \de \PP^Y(y) - b
\end{align}
with positive constants 
\(a := (\int_{\R} \Var(\mathds{1}_{\{Y\geq y\}})\de \PP^Y(y))^{-1}\) and 
\(b := a \int_{\R} \PP(Y<y)^2 \de \PP^Y(y)\); see \cite{ansari2023MFOCI}. 

\item\(\R^2\) fulfills
\begin{align}\label{R2:Rep}
  R^2(Y|\XX) 
  & = \rho_P(Y,Y')
\end{align}
where $\rho_P$ denotes Pearson's correlation coefficient; see \cite{fuchs2023JMVA} and \cite{shih2024}.

\item \(\Lambda\) fulfills
\begin{align}\label{Lambda:Rep}
  \Lambda(Y|\XX)
  = \frac{\PP \big( (Y_1 - Y_2) (Y_1' - Y_2') > 0 \big) 
             - \PP \big( (Y_1 - Y_2) (Y_1' - Y_2') < 0 \big)}{1-\PP(\XX = \XX^\ast)} 
\end{align}
where \((Y_1,Y_1')\) and \((Y_2,Y_2')\) denote independent copies of \((Y,Y')\);
see \cite{fuchs2025}.
\end{enumerate}

From a statistical perspective, the dimension-reduced representations \eqref{Xi:Rep}, \eqref{R2:Rep} and \eqref{Lambda:Rep} via the Markov product have the advantage that the measures of directed dependence \eqref{Eq.Chatterjee}, \eqref{Eq.Pearson} and \eqref{Eq.Fuchs} can be estimated via the nearest neighbor based estimation principle introduced in \cite{chatterjee2021AoS}.
This ensures a strongly consistent, fully non-parametric estimation with no tuning parameters and a computation time of order $O(n \log n)$.

\begin{table}[b]
\begin{center}
\caption{Overview of minimum and maximum value characterizing dependence concepts and their translation into the Markov product.
\label{Overview}}
\begin{tabular}{|l|l|l|}
\hline
Random vector $(\XX,Y)$
& Markov product $(Y,Y')$
& Reference
\\
\hline\hline
$Y$ and $\XX$ are independent  
& $Y$ and $Y'$ are independent
& Theorem \ref{Char:Indep}
\\
\hline
$Y$ is not explainable through $\XX$  
& $Y$ and $Y'$ are uncorrelated
& Corollary \ref{Char:Pearson0}
\\
\hline
$Y$ is stochastic comparable relative to $\XX$  
& The probability of concordance of \((Y,Y')\) equals 
& Corollary \ref{Char:StochComp}
\\
& the probability of discordance of \((Y,Y')\)
&
\\
\hline
$Y$ perfectly depends on $\XX$  
& $Y$ and $Y'$ are comonotone
& Theorem \ref{Thm.Charact:PD}
\\
\hline
$Y$ is completely separated relative to $\XX$  
& $(F_Y(Y),F_Y(Y'))$ admits an ordinal sum structure
& Theorem \ref{Thm.Charact:CS}
\\
\hline
\end{tabular}
\end{center}
\end{table}

Given that each of the above measures of directed dependence admits a Markov product representation, the minimum value (Section \ref{Sec.MinimumValue}) and maximum value (Section \ref{Sec.MaximumValue}) characterizing dependence concepts formulated for the random vector \((\XX,Y)\) can now be translated into dependence concepts for \((Y,Y')\).
This is what this work aims to achieve.
The results are summarized in Table \ref{Overview}; we denote by \(F_Z\) the distribution function of a random variable \(Z\).

\section{Minimum value characterizing dependence concepts} \label{Sec.MinimumValue}

We start with translating independence between \(\XX\) and \(Y\) to the Markov product $(Y,Y')$: 
Theorem \ref{Char:Indep} below verifies that \(\XX\) and \(Y\) are independent if and only if \(Y\) and \(Y'\) are independent.
For random vectors \((\XX,Y)\) having a continuous cdf, the result is due to \cite[Theorem 1]{fuchs2023JMVA}.

\begin{theorem}[Characterizing independence] \label{Char:Indep}
Consider the random vector \((\XX,Y)\) and its Markov product \((Y,Y')\).
Then the following statements are equivalent:
\begin{enumerate}[(a)]
\item \(\xi(Y|\XX)=0\).
\item\label{Indep:XY} \(\XX\) and \(Y\) are independent.
\item\label{Indep:YY'} \(Y\) and \(Y'\) are independent.
\end{enumerate}
\end{theorem}
\begin{proof}
It remains to prove that \eqref{Indep:XY} and \eqref{Indep:YY'} are equivalent. 
We first assume that \eqref{Indep:XY} holds.
Then \eqref{MP:DF} yields
\begin{align*}
\hspace{-10mm}
    \PP(Y\leq y, Y' \leq y') 
    & = \E \big( \PP (Y \leq y \, | \, \XX) \, \PP (Y' \leq y' \, | \, \XX) \big)
      = \E ( \PP (Y \leq y) \, \PP (Y' \leq y') )
      = \PP (Y \leq y) \, \PP (Y' \leq y')
\end{align*}
for all \((y,y') \in \R^2\), hence \(Y\) and \(Y'\) are independent.
Now, assume that \eqref{Indep:YY'} holds.
Then \eqref{MP} and \eqref{MP:DF} yield \linebreak
$\PP(Y\leq y)^2 
 = \PP(Y\leq y, Y'\leq y) 
 = \E \big ( \PP(Y\leq y \, | \, \XX) \, \PP(Y'\leq y \, | \, \XX) \big) 
 = \E \big ( \PP(Y\leq y \, | \, \XX)^2 \big) \geq \E \big ( \PP(Y\leq y \, | \, \XX) \big)^2
      =  \PP(Y\leq y)^2 $ 
for all \(y \in \R\), where the inequality is due to Hölder's inequality.
This then implies 
$$ 0 = \E \big ( \PP(Y\leq y \, | \, \XX)^2 - \PP(Y\leq y)^2 \big)  
     = \E \big ( (\PP(Y\leq y \, | \, \XX) - \PP(Y\leq y))^2 \big) $$
from which we conclude that 
\( \PP(Y \leq y) = \PP(Y\leq y \, | \, \XX = \xx) \) for all $y \in \R$ and \(\PP^\XX\)-almost all \(\xx \in \R^p\).
Thus, \(\XX\) and \(Y\) are independent.
\end{proof}

The next theorem states that $Y$ is not explainable through \(\XX\) if and only if $Y$ and $Y'$ are uncorrelated. The result is immediate from Eq.~\eqref{R2:Rep}.

\begin{corollary}[Characterizing zero-explainability] \label{Char:Pearson0}
Consider the random vector \((\XX,Y)\) and its Markov product \((Y,Y')\).
Then the following statements are equivalent:
\begin{enumerate}[(a)]
\item \(R^2(Y|\XX)=0\).
\item $Y$ is not explainable through \(\XX\).
\item \(Y\) and \(Y'\) are uncorrelated.
\end{enumerate}
\end{corollary}

Finally, we characterize stochastic comparability of $Y$ relative to $\XX$ in terms of the probability of concordance and the probability of discordance of the random vector \((Y,Y')\). The result is immediate from Eq.~\eqref{Lambda:Rep}.

\begin{corollary}[Characterizing stochastic comparability] \label{Char:StochComp}
Consider the random vector \((\XX,Y)\) and its Markov product \((Y,Y')\).
Then the following statements are equivalent:
\begin{enumerate}[(a)]
\item \(\Lambda(Y|\XX)=0\).
\item $Y$ is stochastically comparable relative to $\XX$.
\item The probability of concordance of \((Y,Y')\) coincides with the probability of discordance of \((Y,Y')\), i.e.~\(\PP \big( (Y_1 - Y_2) (Y_1' - Y_2') > 0 \big) = \PP \big( (Y_1 - Y_2) (Y_1' - Y_2') < 0 \big)\) where \((Y_1,Y_1')\) and \((Y_2,Y_2')\) denote independent copies of \((Y,Y')\).
\end{enumerate}
\end{corollary}

For illustrative purposes, we demonstrate how the three aforementioned minimum value characterizing dependence concepts are interrelated.
It is evident that independence of \(\XX\) and \(Y\) implies that $Y$ is not explainable through \(\XX\) and that \(Y\) is stochastically comparable relative to \(\XX\).
Translated into the Markov product, this means that independence of \(Y\) and \(Y'\) implies that \(Y\) and \(Y'\) are uncorrelated and that the probability of concordance of \((Y,Y')\) coincides with the probability of discordance of \((Y,Y')\).
In both cases, the reverse direction does not generally apply (Example \ref{stochasticcomp}), and the two latter dependence concepts are generally not linked (Examples \ref{stochastic} and \ref{Vkick}).

\begin{example}[Neither stochastic comparability nor zero-explainability implies independence]~~ \label{stochasticcomp}
Consider the random variable $X$ with $\PP(X=-1) = 1 - \PP(X=1) = 4/7$ and the random variable $Y$ given by the conditional distributions
\begin{align*}
  \PP^{Y|X=-1}
  & = \mathcal{U}([-1.5, -0.5] \cup [0.5, 1.5])
  \qquad \textrm{ and }
  & \PP^{Y|X=1}
  & = \mathcal{U}([-0.5, 0.5])\,,
\end{align*}
with $\mathcal{U}(A)$ denoting the uniform distribution on a Borel set $A$.
The Markov product \((Y,Y')\) then is uniformly distributed on the intervals \([-1.5, -0.5]^2, [-1.5, -0.5] \times [0.5, 1.5], [0.5, 1.5] \times [-1.5, -0.5], [0.5, 1.5]^2\) each with probability mass $1/7$ and on \([-0.5, 0.5]^2\) with probability mass $3/7$.
Fig.~\ref{fig:stochasticcomp} depicts scatterplots of sample size 1000 for $(X,Y)$ and $(Y,Y')$, drawing attention to the cross-structure of $(Y,Y’)$ and the uniform distributions within the different squares.
For the values of the measures of directed dependence, we obtain
\begin{align*}
  \xi(Y|X) 
  & > 0
  & R^2(Y|X) 
  & = 0
  & \Lambda(Y|X) 
  & = 0\,,
\end{align*}
which indicates that neither $X$ and $Y$ nor $Y$ and $Y'$ are independent.
However, since $R^2(Y|X)=0=\Lambda(Y|X)$ it holds that $Y$ is not explainable through \(X\) and that $Y$ is stochastically comparable relative to $X$.
Translated into the Markov product, this means that \(Y\) and \(Y'\) are uncorrelated and that the probability of concordance of \((Y,Y')\) coincides with the probability of discordance of \((Y,Y')\).
\end{example}

\begin{figure}[t]
\centering
\includegraphics[scale=0.12]{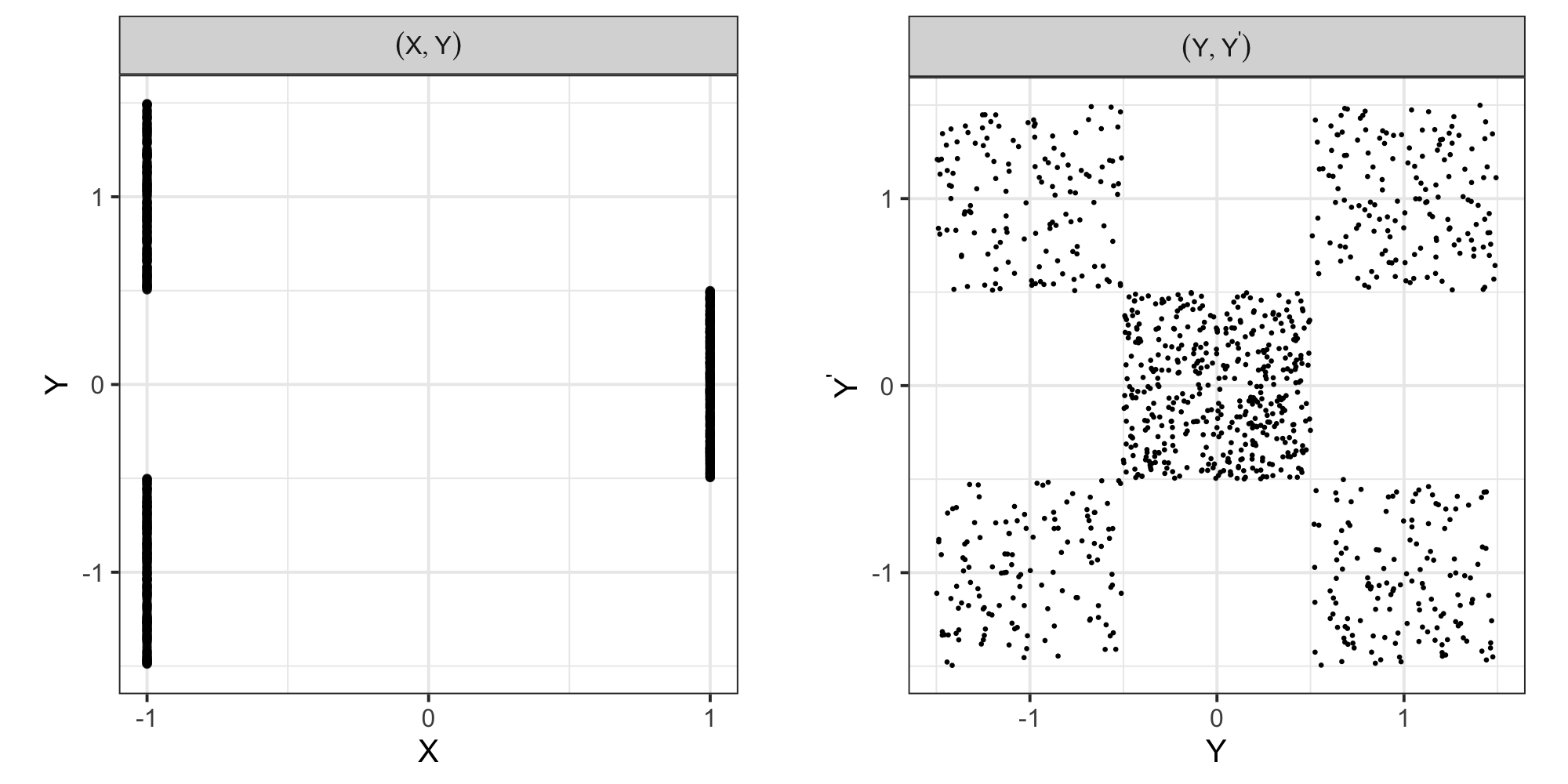}
\caption{Scatterplots of sample size 1000 for the random vector $(X,Y)$ and its Markov product $(Y,Y')$ discussed in Example \ref{stochasticcomp}. 
}
\label{fig:stochasticcomp}
\end{figure}

\begin{example}[Zero-explainability does not imply stochastic comparability]~~ \label{stochastic}
Consider the random variable $X$ with $\PP(X=0) = 1 - \PP(X=1) = 1/3$ and the random variable $Y$ given by the conditional distributions
$\PP^{Y | X=0}= \mathcal{U}([-0.5,0.5])$ and $\PP^{Y | X=1}$,
the latter being a composition of a uniform distribution on $[-1.5, -0.5]$ with probability mass $4/9$ and a uniform distribution on \([1, 3]\) with probability mass \(2/9\).
The Markov product \((Y,Y')\) then is uniformly distributed 
on the interval
\([-1.5, -0.5] \times [1,3], [1,3] \times [-1.5, -0.5]\) each with probability mass $4/27$, 
on \([-1.5, -0.5]^2\) with probability mass $8/27$,  
on \([1,3]^2\) with probability mass $2/27$, and
on \([-0.5,0.5]^2\) with probability mass $1/3$.
Fig.~\ref{stochasticpic} depicts scatterplots of sample size 1000 for $(X,Y)$ and $(Y,Y')$, drawing attention to the uniform distributions within the different rectangles.
For the values of the measures of directed dependence, we obtain
\begin{align*}
  R^2(Y|X) 
  & = 0
  & \Lambda(Y|X) 
  & > 0\,, 
\end{align*}
indicating that $Y$ is not explainable through \(X\), however, $Y$ fails to be stochastically comparable relative to $X$.
Translated into the Markov product, this means that \(Y\) and \(Y'\) are uncorrelated, however, the probability of concordance of \((Y,Y')\) fails to coincides with the probability of discordance of \((Y,Y')\).
\end{example}

\begin{figure}[t]
\centering
\includegraphics[scale=0.12]{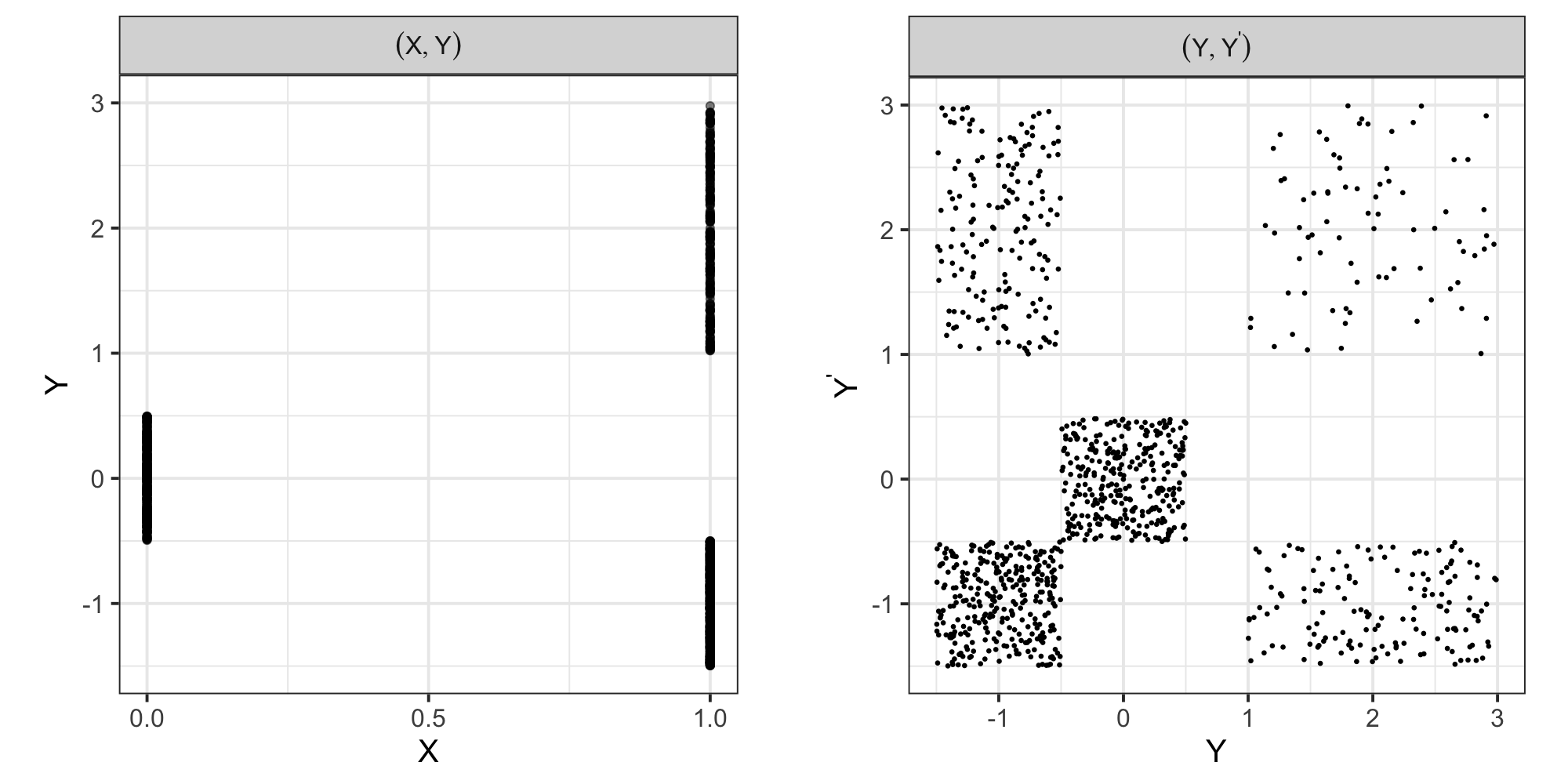}
\caption{Scatterplots of sample size 1000 for the random vector $(X,Y)$ and its Markov product $(Y,Y')$ discussed in Example \ref{stochastic}.
}
\label{stochasticpic}
\end{figure}

\begin{example}[Stochastic comparability does not imply zero-explainability]~~ \label{Vkick}
Consider the random variables $X \sim \mathcal{U}[0,1]$ and 
$Y = 1/2 + 1/2 \cdot Z \cdot X$ where $\PP(Z=-1) = 1/2 = \PP(Z = 1)$, $Z$ being independent of $(X,Y)$.
Then $(Y,Y')$ is cross-shaped.
Fig.~\ref{Vkickpic} depicts scatterplots of sample size 1000 for $(X,Y)$ and $(X,Y^2)$ along with their Markov products.
For the values of the measures of directed dependence, we obtain
\begin{align*}
  R^2(Y|X) 
  & = 0 < R^2(Y^2|X)
  & \Lambda(Y|X) 
  & = 0 = \Lambda(Y^2|X)\,.
\end{align*}
More precisely, \(R^2(Y^2|X) = 1/16\); see \cite[Example 5]{fuchs2023JMVA}.
This indicates that $Y^2$ is stochastically comparable relative to $X$, however, $Y^2$ fails to be not explainable through \(X\).
Translated into the Markov product, this means that the probability of concordance of \((Y^2,(Y^2)')\) coincides with the probability of discordance of \((Y^2,(Y^2)')\), however, \(Y^2\) and \((Y^2)'\) are not uncorrelated.
\end{example}

\begin{figure}[t]
\includegraphics[scale=0.14]{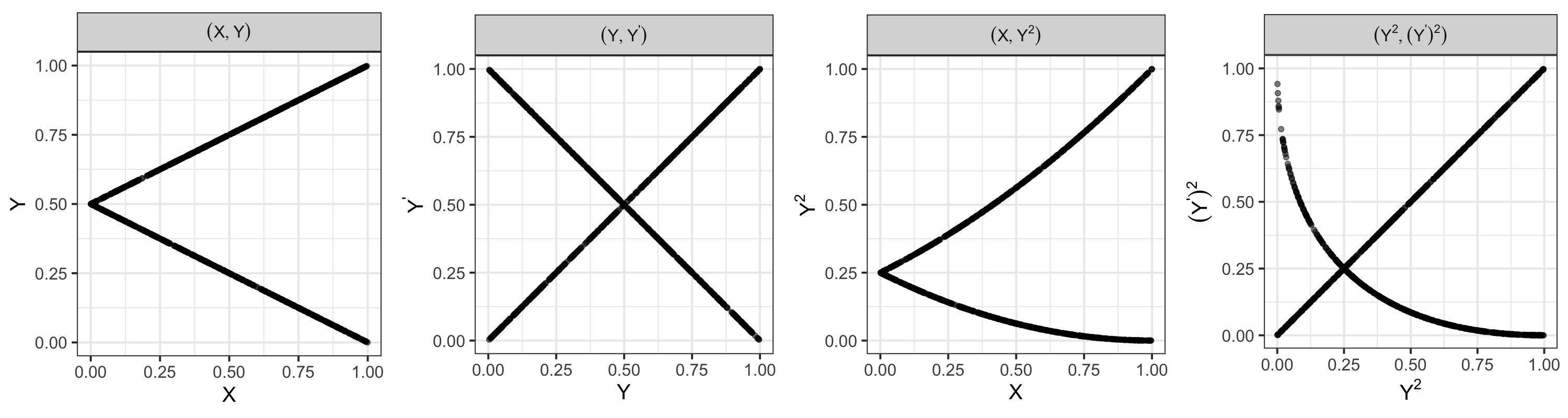}
\caption{Scatterplots of sample size 1000 for the random vectors $(X,Y)$ and $(X,Y^2)$ along with their corresponding Markov products discussed in Example \ref{Vkick}.}
\label{Vkickpic}
\label{Vkickpic2}
\end{figure}

\section{Maximum value characterizing dependence concepts} \label{Sec.MaximumValue}

We first present a characterization of perfect dependence in terms of the Markov product: 
Theorem \ref{Thm.Charact:PD} below states that 
\(Y\) perfectly depends on \(\XX\) if and only if 
\(Y\) and \(Y'\) are comonotone if and only if 
\(Y = Y'\) almost surely.
The last equivalence is remarkable in that comonotonicity, i.e~perfect monotone dependence, together with $Y \eqd Y'$ implies equality almost surely.
For random vectors \((\XX,Y)\) having a continuous cdf, the first equivalence is due to \cite[Theorem 1]{fuchs2023JMVA}.

\begin{theorem}[Characterizing perfect dependence]\label{Thm.Charact:PD}
Consider the random vector \((\XX,Y)\) and its Markov product \((Y,Y')\).
Then the following statements are equivalent:
\begin{enumerate}[(a)]
\item \(\xi(Y|\XX) = R^2(Y|\XX) = 1\).
\item\label{PD:XY} \(Y\) perfectly depends on \(\XX\).
\item\label{PD:YY'} \(Y\) and \(Y'\) are comonotone.
\item\label{PD:YY'2} \(Y = Y'\) almost surely.
\end{enumerate}
\end{theorem}

\begin{proof}
It remains to prove that \eqref{PD:XY}, \eqref{PD:YY'} and \eqref{PD:YY'2} are equivalent. 

We first assume that \eqref{PD:XY} holds. 
Then \(Y = f(\XX)\) almost surely for some measurable function \(f\), hence \(\PP(Y\leq y \, | \, \XX) = \mathds{1}_{(-\infty,y]} (f(\XX))\), and \eqref{MP} and \eqref{MP:DF} yield
\begin{align*}
  \PP(Y \leq y, Y' \leq y') 
  & = \E \big ( \PP(Y\leq y \, | \, \XX) \, \PP(Y'\leq y' \, | \, \XX) \big) 
    = \E \big ( \mathds{1}_{(-\infty,y]} (f(\XX)) \, \mathds{1}_{(-\infty,y']} (f(\XX)) \big) 
  \\
  & = \E \big ( \mathds{1}_{(-\infty,\min\{y,y'\}]} (f(\XX)) \big)
    = \E \big ( \PP(Y\leq \min\{y,y'\} \, | \, \XX) \big) 
  \\
  & = \PP(Y\leq \min\{y,y'\}) 
    = \min\{F_Y(y), F_Y(y')\}
    = \min\{F_Y(y), F_{Y'}(y')\}
\end{align*}
for all \((y,y') \in \mathbb{R}^2\), where the last identity is due to the fact that $Y \eqd Y'$. Thus, \(Y\) and \(Y'\) are comonotone. 

We now prove that \eqref{PD:YY'} implies \eqref{PD:XY}. 
Using \eqref{MP:DF} again gives
\begin{align*}
\hspace{-10mm}
  \PP (Y \leq y) 
  & = \min\{F_Y(y), F_{Y'}(y)\} 
    = \PP(Y \leq y, Y' \leq y) 
    = \E \big ( \PP(Y\leq y \, | \, \XX)^2 \big) 
   \leq \E \big ( \PP(Y\leq y \, | \, \XX) \big) 
    = \PP(Y \leq y) 
\end{align*}
for all \(y \in \mathbb{R}\).
This then implies 
$$ 0 = \E ( \PP(Y\leq y \, | \, \XX) - \PP(Y\leq y \, | \, \XX)^2 ) 
    = \E ( \PP(Y\leq y \, | \, \XX) \, (1 - \PP(Y \leq y \, | \, \XX)) ) $$ and hence
\begin{align} \label{Thm.Charact:PD.Eq}
  \PP(Y\leq y \, | \, \XX=\xx) \in \{0,1\}    
\end{align}
 for all $y \in \R$ and \(\PP^\XX\)-almost all \(\xx \in \R^p\).
Now, for \(\xx \in G\) with \(\PP^\XX(G) = 1\), define 
\begin{align*}
    a_\xx   
    &:= \sup \left\{ y \in \mathbb{R} \, : \, \PP^{Y|\XX=\xx} ((-\infty,y]) = 0 \right\}
    & b_\xx 
    &:= \inf \left\{ y \in \mathbb{R} \, : \, \PP^{Y|\XX=\xx} ((-\infty,y]) = 1 \right\}\,,
\end{align*}
so $a_\xx \leq b_\xx$. Assume $a_\xx < b_\xx$.
Then either \(\PP(Y \leq a_\xx \,|\, \XX=\xx) = 0\) and
\(\PP(Y \leq z \,|\, \XX=\xx) \in (0,1)\) for all \(z \in (a_\xx,b_\xx)\) and hence
\(\PP(Y \in (a_\xx,b_\xx) \,|\, \XX=\xx) > 0\),
or \(\PP(Y \leq z \,|\, \XX=\xx) \in (0,1)\) for all \(z \in [a_\xx,b_\xx)\) and hence
\(\PP(Y \in [a_\xx,b_\xx) \,|\, \XX=\xx) > 0\). Both cases contradict \eqref{Thm.Charact:PD.Eq}.
Therefore, $a_\xx = b_\xx$.
Thus, for almost every \(\xx \in \mathbb{R}^p\) there exists some constant \(c_\xx \in \mathbb{R}\) such that \(\PP(Y = c_\xx\,|\, \XX=\xx) = 1\), i.e.~\(Y\) is almost surely a function of \(\XX\).

Finally, \eqref{PD:YY'2} clearly implies \eqref{PD:YY'},
and the converse direction is due to the fact that \eqref{PD:YY'} implies \(\PP^{(Y,Y')} = \PP^{T(Y)}\) with $T(z):= (z,z)$ from which 
\begin{align*}
    \PP(Y=Y')
    & 
      =  \int_{\R^2} \mathds{1}_{\{(z,z') \in \R^2: z = z'\}} (y,y') \de \PP^{T(Y)}(y,y')
      =  \int_{\R} \mathds{1}_{\{(z,z') \in \R^2: z = z'\}} (y,y) \de \PP^{Y}(y)
      = 1
\end{align*}
immediately follows, where the second identity is due to change of coordinates.
This proves \eqref{PD:YY'2}.
\end{proof}

We proceed with a characterization of complete separation in terms of ordinal sum structures on $[0,1]^2$. 
Since $\Lambda(Y|\XX)$ remains unchanged when replacing the random variable $Y$ by its individual distributional transform due to \cite[Proposition 2.8]{fuchs2025}, 
i.e.~$\Lambda(Y|\XX) = \Lambda(F_Y(Y) \,|\, \XX)$, in what follows we work with $(\XX,F_Y(Y))$ and its Markov product \((F_Y(Y),F_Y(Y'))\). 
We note in passing that the mentioned invariance also applies to $\xi$ \cite[Proposition 2.4]{ansari2023MFOCI} but not to $R^2$ (recall Example \ref{Vkick}).

Inspired by the definition of an ordinal sum developed for copulas (see, e.g., \cite{durante2016}), we here propose a generalization for arbitrary distributions on $[0,1]^2$:
Let \(N\) be a finite or countably finite subset of the natural numbers. 
Further, let \(\{(a_k,b_k]\}_{k \in N}\) be a family of non-overlapping subintervals of \([0,1]\) with $0 \leq a_k < b_k \leq 1$ for all \(k \in N\).
Then the \emph{ordinal sum} \(F = (\langle (a_k,b_k], F_k \rangle)_{k \in N}\) of size $|N|$ for a sequence of $2$-dimensional distribution functions \((F_k)_{k \in N}\), each \(F_k\) having support in $(a_k,b_k]^2$, $k \in N$, is defined, for \((u,v) \in [0,1]^2\), by 
\begin{align}\label{Def:OSum}
  F(u,v) 
  & := \sum_{k \in N} (b_k - a_k) \, F_k (u,v)
       + \lambda\left( [0,\min\{u,v\}] \backslash \bigcup_{k \in N} (a_k,b_k] \right] \,.
\end{align}
The simplest ordinal sum structure is the one of order $0$, i.e.~$F(u,v)=\min\{u,v\}$ for all $(u,v) \in [0,1]^2$. This refers to a comonotonic distribution function on $[0,1]^2$. 
An ordinal sum representation is generally not unique as illustrated in Fig.~\ref{OSP}.
\begin{figure}[t]
  \centering
  \includegraphics[width=0.90\linewidth]{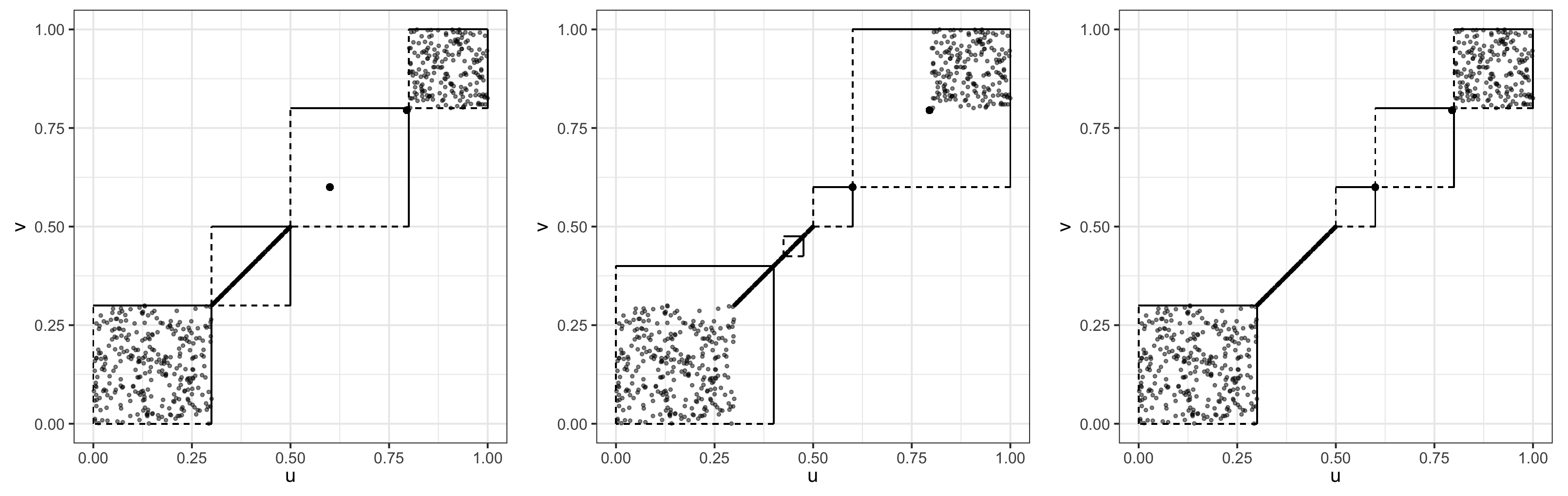}
  \caption{Scatterplot of sample size $2000$ of a distribution on $[0,1]^2$ with three different suitable ordinal sum representations.} 
  \label{OSP}
\end{figure}
In fact, every distribution function $F = (\langle (0,1], F \rangle)_{k \in N}$ with $N = \{1\}$ is an ordinal sum of size $1$ (and hence of trivial structure). 

Theorem \ref{Thm.Charact:CS} below states that complete separation of $Y$ relative to $\XX$ translates into a specific ordinal sum structure for $(F_Y(Y),F_Y(Y'))$ with the size of the ordinal sum referring to the number of discrete points of the vector $\XX$.

\begin{theorem}[Characterizing complete separation] \label{Thm.Charact:CS}
Consider the random vectors $(\XX,Y)$, $(\XX,F_Y(Y))$ and its Markov product $(F_Y(Y),F_Y(Y'))$.
Define $M := \{\zz \in \mathbb{R}^p \,:\, \PP(\XX=\zz)>0\}$.
Then the following statements are equivalent:
\begin{enumerate}[(a)]
\item \(\Lambda(Y|\XX) = 1\).
\item\label{CS:XY} \(Y\) is completely separated relative to \(\XX\).
\item\label{CS:YY'} \((F_Y(Y),F_Y(Y'))\) admits an ordinal sum structure \((\langle (a_\zz,b_\zz], F_\zz \rangle)_{\zz \in M}\) of size $|M|$ such that, for every $\zz \in M$, $b_\zz - a_\zz = \PP(\XX=\zz)$ and $F_\zz(u,v) = F_\zz(u,b_\zz) \cdot F_\zz(b_\zz,v)$ for all $(u,v) \in (a_\zz,b_\zz]^2$.
\end{enumerate}
\end{theorem}
\begin{proof}
It remains to prove that \eqref{CS:XY} and \eqref{CS:YY'} are equivalent.
We note in passing that $F_{F_Y(Y)}(u) = $ \linebreak $\PP(F_Y(Y) \leq u) \leq u$ for all $u \in [0,1]$ and
\begin{align}\label{Thm.Charact:CS.5}
    \PP(F_Y(Y) \leq u) = u
\end{align}
for all $u \in \supp(\PP^{F_Y(Y)})$, the support of $\PP^{F_Y(Y)}$.

Assume first that \(Y\) is completely separated relative to \(\XX\).
Then \(F_Y(Y)\) is completely separated relative to \(\XX\) due to \cite[Proposition 2.8]{fuchs2025}, i.e.~the relative effect 
\begin{align}\label{Thm.Charact:CS.1}
  \Psi \big(\PP^{F_Y(Y)|\XX=\xx_1}, \PP^{F_Y(Y)|\XX=\xx_2}\big) \in \{0,1\}
\end{align}
for almost all $(\xx_1,\xx_2)$ with $\xx_1 \neq \xx_2$. \pagebreak
\\
In a first step, we determine the sequence \(\{(a_\zz,b_\zz]\}_{\zz \in M}\) of non-overlapping intervals.
For $\xx \in G := \supp(\PP^\XX)$ denote by 
\begin{align*}
  l(\xx)
  & := \inf(\supp(\PP^{F_Y(Y) |\XX = \xx})) 
  & u(\xx)
  & := \sup(\supp(\PP^{F_Y(Y) |\XX = \xx}))   
\end{align*}
the infimum and supremum of the support of $\PP^{F_Y(Y) |\XX = \xx}$.
Then \eqref{Thm.Charact:CS.1} implies  
\begin{align}\label{Thm.Charact:CS.1A}
    (l(\xx_1),u(\xx_1)) \cap (l(\xx_2),u(\xx_2)) = \emptyset
\end{align}
for almost all $\xx_1,\xx_2 \in G$ with $\xx_1 \neq \xx_2$.
We distinguish two cases: 
\begin{enumerate}
    \item 
    For $\zz \in M \subseteq G$, set $b_\zz := u(\zz)$. 
    If there exists some $\xx \in G$ such that $u(\ww) \leq u(\xx) \leq l(\zz)$ for all $\ww \in G$, then define $a_\zz := u(\xx)$. Otherwise, set $a_\zz := 0$.
    Then, according to \eqref{Thm.Charact:CS.1A}, 
    \begin{align}\label{Thm.Charact:CS.2}
        (a_{\zz_1},b_{\zz_1}] \cap (a_{\zz_2},b_{\zz_2}] = \emptyset 
    \end{align} 
    for all $\zz_1,\zz_2 \in M$ with $\zz_1 \neq \zz_2$.
    
    \item 
    \eqref{Thm.Charact:CS.1A} further implies that for almost all $\xx \in G \backslash M$ the conditional distributions $\PP^{F_Y(Y) |\XX = \xx}$ are degenerate with \(\PP^{F_Y(Y) |\XX = \xx}(\{u(\xx)\})=1\); 
    otherwise this would contradict the fact that $\PP(\XX=\xx) = 0$ for all $\xx \in G\backslash M$. 
    Thus, there exists some function $u: G \backslash M \to [0,1]\backslash \bigcup_{\zz \in M} (a_\zz,b_\zz]$ such that $F_Y(Y) = u(\XX)$. 
    According to \eqref{Thm.Charact:CS.1}
    \begin{align}\label{Thm.Charact:CS.3}
      \qquad u(\xx_1) \neq u(\xx_2)
    \end{align} 
    for almost all $\xx_1,\xx_2 \in G \backslash M$ with $\xx_1 \neq \xx_2$.
\end{enumerate}
It then follows from \eqref{Thm.Charact:CS.1} that also 
\begin{align}\label{Thm.Charact:CS.4} 
  (a_\zz,b_\zz] \cap \{u(\xx)\} = \emptyset
\end{align} 
for all $\zz \in M$ and almost all $\xx \in  G \backslash M$.
\\
In a second step, we verify the ordinal sum structure.
For $\zz \in M$, disintegration together with \eqref{MP} yields
\begin{align*}
    \PP^{(F_Y(Y),F_Y(Y'))} ((a_\zz,b_\zz]^2)
    & = \int_{\R^p} \PP^{(F_Y(Y),F_Y(Y')) |\XX = \xx} ((a_\zz,b_\zz]^2) \de \PP^{\XX}(\xx) \notag
    \\
    & = \int_{\R^p} \PP^{F_Y(Y)|\XX = \xx} ((a_\zz,b_\zz])^2 \de \PP^{\XX}(\xx) \notag
    \\
    & = \PP(\XX = \zz) \, \PP^{F_Y(Y)|\XX = \zz} ((a_\zz,b_\zz])^2 
      = \PP(\XX = \zz)\,,
\end{align*}
where the second last identity is due to \eqref{Thm.Charact:CS.2}, \eqref{Thm.Charact:CS.3} and \eqref{Thm.Charact:CS.4}, and the last identity follows from \linebreak \(\PP^{F_Y(Y)|\XX = \zz} ((a_\zz,b_\zz]) = 1\).
For the same reason \(\PP^{F_Y(Y)} ((a_\zz,b_\zz]) = \PP(\XX = \zz)\) and hence 
$ \PP(\XX = \zz)
    = \PP^{F_Y(Y)} ((a_\zz,b_\zz])
    = \PP(F_Y(Y) \leq b_\zz) - \PP(F_Y(Y) \leq a_\zz)
    = b_\zz - a_\zz$,
where the last identity is due to \eqref{Thm.Charact:CS.5}.
For all $(u,v) \in [0,1]^2$, disintegration and \eqref{MP} then yield 
\begin{align} \label{Thm.Charact:CS.6} 
  &   \PP(F_Y(Y) \leq u, F_Y(Y') \leq v)
    = \int_{\R^p} \PP(F_Y(Y) \leq u \,|\, \XX = \xx) \, \PP(F_Y(Y) \leq v \,|\, \XX = \xx) \de \PP^{\XX}(\xx) \notag
  \\
  & = \sum_{\zz \in M} \underbrace{\PP(\XX = \zz)}_{= b_\zz - a_\zz} \, \PP(F_Y(Y) \leq u \,|\, \XX = \zz) \, \PP(F_Y(Y) \leq v \,|\,\XX = \zz) \notag
  \\
  & \qquad\qquad  + \int_{G \backslash M} 
              \underbrace{\PP(F_Y(Y) \leq u \,|\, \XX = \xx)}_{\eqref{Thm.Charact:CS.3} = \mathds{1}_{[0,u]}(u(\xx))} \, 
              \underbrace{\PP(F_Y(Y) \leq v \,|\, \XX = \xx)}_{\eqref{Thm.Charact:CS.3} = \mathds{1}_{[0,v]}(u(\xx))} \de \PP^{\XX}(\xx) \notag
  \\
  & = \sum_{\zz \in M} (b_\zz - a_\zz) \, \PP(F_Y(Y) \leq u \,|\, \XX = \zz) \, \PP(F_Y(Y) \leq v \,|\,\XX = \zz) 
       + \int_{G \backslash M} \mathds{1}_{[0,\min\{u,v\}]}(u(\xx)) \de \PP^{\XX}(\xx)\,.
\end{align}
Setting \(F_\zz(u,v) := \PP(F_Y(Y) \leq u \,|\, \XX = \zz) \, \PP(F_Y(Y) \leq v \,|\,\XX = \zz)\) for all $(u,v) \in [0,1]^2$ and all $\zz \in M$ yields a sequence of $2$-dimensional distribution functions each with support in $(a_\zz, b_\zz]^2$ and such that \(F_\zz(u,v) = F_\zz(u,b_\zz) \cdot F_\zz(b_\zz,v)\) for all $(u,v) \in (a_\zz,b_\zz]^2$. This gives the first term in \eqref{Def:OSum}.
For the second term, we first observe that for all $ u \in [0,1]\backslash \bigcup_{\zz \in M} (a_\zz,b_\zz]$ we have $\PP(F_Y(Y) \leq u) = u = \lambda([0,u])$ due to \eqref{Thm.Charact:CS.5}.
Then, for every $u \in [0,1]\backslash \bigcup_{\zz \in M} (a_\zz,b_\zz]$, 
from \eqref{Thm.Charact:CS.6} and change of coordinates we obtain
\begin{align} \label{Thm.Charact:CS.8} 
  & \lambda \left( [0,u] \backslash \bigcup_{\zz \in M} (a_\zz,b_\zz] \right)
    = \lambda([0,u]) - \lambda \left( \bigcup_{\zz \in M, b_\zz < u} (a_\zz,b_\zz] \right)
  \notag
  \\
  & = \PP(F_Y(Y) \leq u) - \lambda \left( \bigcup_{\zz \in M, b_\zz < u} (a_\zz,b_\zz] \right)
  \notag
  \\
  & = \sum_{\zz \in M, b_\zz < u} \underbrace{(b_\zz - a_\zz)}_{= \lambda ((a_\zz,b_\zz])} \, \underbrace{\PP(F_Y(Y) \leq u \,|\, \XX = \zz)}_{=1}
       + \int_{G \backslash M} \mathds{1}_{[0,u]}(u(\xx)) \de \PP^{\XX}(\xx)
       - \lambda \left( \bigcup_{\zz \in M, b_\zz < u} (a_\zz,b_\zz] \right)
  \notag
  \\
  & = \lambda \left( \bigcup_{\zz \in M, b_\zz < u} (a_\zz,b_\zz] \right) \,  
       + \int_{[0,1]\backslash \bigcup_{\zz \in M} (a_\zz,b_\zz]} \mathds{1}_{[0,u]}(t) \de \underbrace{\PP^{u(\XX)}}_{ \eqref{Thm.Charact:CS.3} = \PP^{F_Y(Y)}}(\xx)
       - \lambda \left( \bigcup_{\zz \in M, b_\zz < u} (a_\zz,b_\zz] \right)
  \nonumber
  \\
  & = \PP^{F_Y(Y)} \left( [0,u]\backslash \bigcup_{\zz \in M} (a_\zz,b_\zz] \right)\,.
\end{align}  
If, instead, $u \in \bigcup_{\zz \in M} (a_\zz,b_\zz]$ then there exists a largest $u^\ast \in [0,1]\backslash \bigcup_{\zz \in M} (a_\zz,b_\zz]$ such that $u^\ast \leq u$ and hence 
\begin{align} \label{Thm.Charact:CS.9} 
  \lambda \left( [0,u] \backslash \bigcup_{\zz \in M} (a_\zz,b_\zz] \right)
  & = \lambda \left( [0,u^\ast] \backslash \bigcup_{\zz \in M} (a_\zz,b_\zz] \right)
  \notag
  \\ 
  & = \PP^{F_Y(Y)} \left( [0,u^\ast]\backslash \bigcup_{\zz \in M} (a_\zz,b_\zz] \right)
    = \PP^{F_Y(Y)} \left( [0,u]\backslash \bigcup_{\zz \in M} (a_\zz,b_\zz] \right)\,.
\end{align}  
For the second term in \eqref{Thm.Charact:CS.6} we thus have
\begin{align*}
  \int_{G \backslash M} \mathds{1}_{[0,\min\{u,v\}]}(u(\xx)) \de \PP^{\XX}(\xx)
  & = \int_{[0,1]\backslash \bigcup_{\zz \in M} (a_\zz,b_\zz]} \mathds{1}_{[0,\min\{u,v\}]}(t) \de \PP^{u(\XX)}(t)
  \\
  & = \int_{[0,1]\backslash \bigcup_{\zz \in M} (a_\zz,b_\zz]} \mathds{1}_{[0,\min\{u,v\}]}(t) \de \PP^{F_Y(Y)}(t)
  \\
  & = \lambda \left( [0,\min\{u,v\}]\backslash \bigcup_{\zz \in M} (a_\zz,b_\zz] \right)\,,
\end{align*}
where we use change of coordinates together with \eqref{Thm.Charact:CS.8} and \eqref{Thm.Charact:CS.9}.
Therefore, \((F_Y(Y),F_Y(Y'))\) admits an ordinal sum structure \((\langle (a_\zz,b_\zz], F_\zz \rangle)_{\zz \in M}\) of size $|M|$ such that, for every $\zz \in M$, $b_\zz - a_\zz = \PP(\XX=\zz)$ and $F_\zz(u,v) = F_\zz(u,b_\zz) \, F_\zz(b_\zz,v)$ for all $(u,v) \in (a_\zz,b_\zz]^2$.

Now, assume that \eqref{CS:YY'} holds.
In what follows we verify that the supports of the conditional distributions are separated.
For $\zz \in M$, a direct use of the ordinal sum structure first gives
$ \PP((F_Y(Y),F_Y(Y')) \in (a_\zz,b_\zz]^2) \linebreak 
  = \PP(\XX=\zz) 
  = \PP(F_Y(Y) \in (a_\zz,b_\zz])$,
and disintegration together with \eqref{MP} then yields
\begin{align*}
  0 
  & = \PP(F_Y(Y) \in (a_\zz,b_\zz]) - \PP((F_Y(Y),F_Y(Y')) \in (a_\zz,b_\zz]^2)
  \\
  & = \int_{\R^p} \PP(F_Y(Y) \in (a_\zz,b_\zz] | \XX=\xx) - \PP((F_Y(Y),F_Y(Y')) \in (a_\zz,b_\zz]^2 | \XX=\xx) \de \PP^\XX(\xx)
  \\
  & = \int_{\R^p} \PP(F_Y(Y) \in (a_\zz,b_\zz] | \XX=\xx) - \PP(F_Y(Y) \in (a_\zz,b_\zz] | \XX=\xx)^2 \de \PP^\XX(\xx)\,.
\end{align*}
Hence \(\PP(F_Y(Y) \in (a_\zz,b_\zz] | \XX=\xx) \in \{0,1\}\) for $\PP^\XX$-almost all $\xx \in \R^p$.
Since
$ 0 
  < \PP(\XX=\zz) \linebreak
  = \PP(F_Y(Y) \in (a_\zz,b_\zz])
  = \PP(F_Y(Y) \in (a_\zz,b_\zz] | \XX=\zz) \, \PP(\XX=\zz)$,
it follows that \(\PP(F_Y(Y) \in (a_\zz,b_\zz] | \XX=\zz) = 1\) for all $\zz \in M$.
\\ 
We now consider the set $G \backslash M$ with $G := \supp(\PP^\XX)$ and verify that the supports of the conditional distributions of $F_Y(Y)$ given $\XX=\xx$ for almost all $\xx \in G \backslash M$ are also separated. 
First, notice that by assumption
\begin{align*}
    \PP \left(F_Y(Y) \in [0,1] \backslash \bigcup_{\zz \in M} (a_\zz,b_\zz] \right)
    & = \PP (\XX \in G \backslash M)\,,
\end{align*}
and consider $u \in [0,1] \backslash \bigcup_{\zz \in M} (a_\zz,b_\zz]$.
Then, using again the ordinal sum structure gives \linebreak
$ \PP(F_Y(Y) \leq u, F_Y(Y') \leq u) 
  = \PP(F_Y(Y) \leq u) 
  = u$,
and disintegration together with \eqref{MP} then yields
\begin{align*}
  0 
  & = \PP(F_Y(Y) \leq u) - \PP(F_Y(Y) \leq u, F_Y(Y') \leq u) 
  \\
  & = \int_{\R^p} \PP(F_Y(Y) \leq u \,|\, \XX=\xx) - \PP(F_Y(Y) \leq u, F_Y(Y') \leq u \,|\, \XX=\xx) \de \PP^\XX(\xx)
  \\
  & = \int_{\R^p} \PP(F_Y(Y) \leq u \,|\, \XX=\xx) - \PP(F_Y(Y) \leq u \,|\, \XX=\xx)^2 \de \PP^\XX(\xx)\,.
\end{align*}
Hence \(\PP(F_Y(Y) \leq u \,|\, \XX=\xx) \in \{0,1\}\) for $\PP^\XX$-almost all $\xx \in \R^p$.
Now, similar to what has been shown in the proof of Theorem \ref{Thm.Charact:PD} \eqref{PD:YY'} to \eqref{PD:XY} we conclude that for almost all \(\xx \in G \backslash M\) there exists some constant \(u_\xx \in [0,1] \backslash \bigcup_{\zz \in M} (a_\zz,b_\zz]\) such that \(\PP(F_Y(Y) = u_\xx\,|\, \XX=\xx) = 1\). And it follows from the ordinal sum structure that $u_{\xx_1} \neq u_{\xx_2}$ for almost all $\xx_1,\xx_2 \in G \backslash M$ with $\xx_1 \neq \xx_2$.
Therefore, the supports of the conditional distributions of $F_Y(Y)$ given $\XX=\xx$ for almost all $\xx \in G$ are separated, and thus, $F_Y(Y)$ is completely separated relative to $\XX$, but this is equivalent to $Y$ being completely separated relative to $\XX$ due to \cite[Proposition 2.8]{fuchs2025}.
\end{proof}

To illustrate our findings, we demonstrate how the two aforementioned maximum value characterizing dependence concepts are interrelated:
As shown in \cite{fuchs2025} in terms of the random vector $(\XX,Y)$, complete separation and perfect dependence are generally not connected.  
To complete the picture, in Examples \ref{Trans} and \ref{Ex.REM.PS.vs.PD} we now illustrate the general lack of connection between perfect dependence and complete separation also by means of the Markov product.

\begin{example}[Complete separation does not imply perfect dependence]~~ \label{Trans}
Consider the random variable $X$ with $\PP(X=-1) = 1 - \PP(X=1) = 1/3$ and the random variable $Y$ given by the conditional distributions 
\begin{align*}
    P^{Y|X=-1}=\mathcal{U}[-1,0]
    \qquad 
    \textrm{ and } 
    \qquad 
    P^{Y|X=1} = Z\cdot\mathcal{U}[1,3]+(1-Z) 
\end{align*}
with $Z$ being independent of $(X,Y)$ and $\PP(Z=0) = 1/2 = \PP(Z=1)$.
Fig.~\ref{E3P} depicts scatterplots of sample size 1000 for the random vectors $(X,Y)$, $(Y,Y')$, $(X,F_Y(Y))$ and $(F_Y(Y),F_Y(Y'))$.
Clearly, 
$(F_Y(Y),F_Y(Y'))$ admits an ordinal sum structure of size $2 = |\{z \in \mathbb{R} \,:\, \PP(X=z)>0\}|$.
However, $Y$ and $Y'$ are not equal almost surely.
Therefore, $Y$ is completely separated relative to $X$ but fails to be perfectly dependent on $X$. 
\end{example}

\begin{figure}[t]
  \centering
  \includegraphics[scale=0.14]{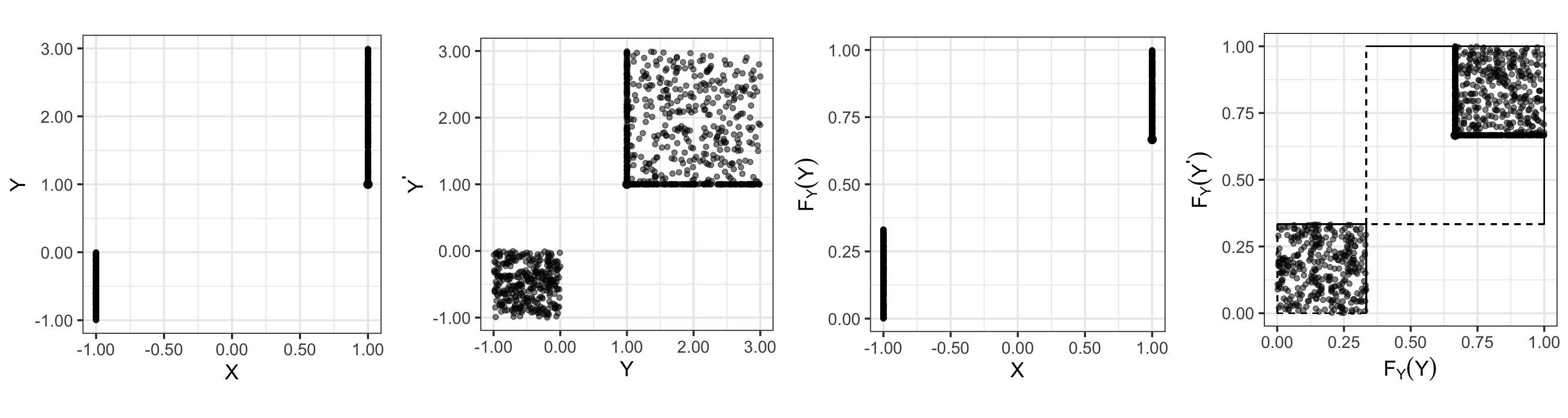}
  \caption{Scatterplots of sample size $1000$ showing $(X, Y)$, its Markov product $(Y,Y')$ and the transformation $(X, F_Y(Y))$ together with its corresponding Markov product $(F_Y(Y), F_Y(Y'))$ discussed in Example \ref{Trans}.}
  \label{E3P}
\end{figure}

\begin{example}[Perfect dependence does not imply complete separation]~~ \label{Ex.REM.PS.vs.PD}
Consider the random variables $X \sim \mathcal{U}([-1,0]\cup [1,3])$ and $Y = 3 \cdot \mathds{1}_{\{X \geq 1\}} - 1$. 
Fig.~\ref{E1P} depicts scatterplots of sample size $1000$ for the vectors $(X,Y)$, $(Y,Y')$, $(X,F_Y(Y))$ and $(F_Y(Y),F_Y(Y'))$.
Clearly, 
$Y=Y'$ almost surely.
However, $(F_Y(Y),F_Y(Y'))$ fails to admit an ordinal sum structure of size $0$, which is (due to the fact that $X$ has a continuous cdf) the only permissible ordinal sum structure for $Y$ to be stochastically comparable relative to $X$ according to Theorem \ref{Thm.Charact:CS}.
Therefore, $Y$ perfectly depends on $X$ but fails to be completely separated relative to $X$. 
\end{example}

\begin{figure}[t]
  \centering
  \includegraphics[scale=0.14]{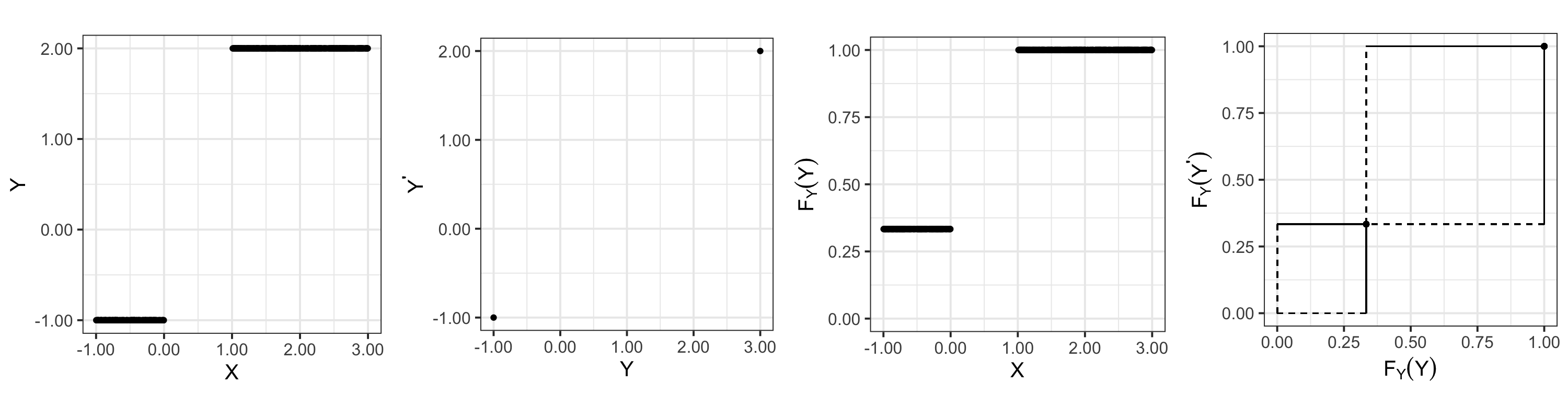}
  \caption{Scatterplots of sample size $1000$ showing $(X, Y)$, its Markov product $(Y,Y')$ and the transformation $(X, F_Y(Y))$ together with its corresponding Markov product $(F_Y(Y), F_Y(Y'))$ discussed in Example \ref{Ex.REM.PS.vs.PD}.}
  \label{E1P}
\end{figure}

If $Y$ has a continuous cdf, then perfect dependence implies complete separation as shown in Theorem 2.12 in \cite{fuchs2025}.
In contrast, if at least one of the coordinates of $\XX$ has a continuous cdf, Theorem \ref{Thm.Charact:CS} simplifies and complete separation implies perfect dependence.
The next result hence complements Theorem 2.11 in \cite{fuchs2025}.

\begin{corollary}[Characterizing complete separation] \label{Cor.Charact:CS}
Consider the random vectors $(\XX,Y)$, $(\XX,F_Y(Y))$ and its Markov product $(F_Y(Y),F_Y(Y'))$, and suppose that one of the coordinates of $\XX$ has a continuous cdf.
Then $\{\zz \in \mathbb{R}^p \,:\, \PP(\XX=\zz)>0\} = \emptyset$ and the following statements are equivalent:
\begin{enumerate}[(a)]
\item \(\Lambda(Y|\XX) = 1\).
\item \(Y\) is completely separated relative to \(\XX\).
\item \label{Cor.Charact:CS.c} \(\PP(F_Y(Y) \leq u, F_Y(Y') \leq v) = \min\{u,v\} = \lambda ([0,\min\{u,v\}])\) for all $(u,v) \in [0,1]^2$,
i.e.~$(F_Y(Y),F_Y(Y'))$ admits an ordinal sum structure of size $0$.
\end{enumerate}
In either case $F_Y(Y)$ and $F_Y(Y')$ are comonotone, 
and thus, $Y$ perfectly depends on $\XX$ with \(\xi(Y|\XX) = R^2(Y|\XX) = 1\).
\end{corollary}
\begin{proof}
The equivalence is immediate from Theorem \ref{Thm.Charact:CS}.
Moreover, from \eqref{Cor.Charact:CS.c} it follows that $F_Y(Y)$ and $F_Y(Y')$ are comonotone, and \cite[Proposition 2.4]{ansari2023MFOCI} together with Theorem \ref{Thm.Charact:PD} then gives $\xi(Y|\XX) = \xi(F_Y(Y)|\XX) = 1$, hence $Y$ perfectly depends on $\XX$, and thus, $R^2(Y|\XX) = 1$.
\end{proof}


To round off the discussion, we conclude with an example $(X,Y)$ where $Y$ is completely separated relative to $\XX$ and $(F_Y(Y),F_Y(Y'))$ admits an ordinal sum structure of infinite size.


\begin{example}\label{Kette}
Consider the random variable $X$ with 
$\PP(X = n) = \frac{1}{2^n}$, $n \in \mathbb{N}$, and the random variable $Y$ given by the conditional distributions 
\begin{align*}
  \PP^{Y|X=n} 
  & = \mathcal{U}\left(2- 2^{-n+3}, 2- 2^{-n+2}\right)\,.
\end{align*}
Then
\begin{align*}
   \PP^{(Y,Y')} 
   & = \sum_{n \in \mathbb{N}} \frac{1}{2^n} \; 
       \mathcal{U}\left( \left(2- 2^{-n+3},2- 2^{-n+2}\right)^2 \right) 
\end{align*}
and $(F_Y(Y),F_Y(Y'))$ admits an ordinal sum structure of infinite size.
Fig.~\ref{E2P} depicts scatterplots of sample size $1000$ for the random vectors $(X,Y)$, $(Y,Y')$, $(X,F_Y(Y))$ and $(F_Y(Y),F_Y(Y'))$.
According to Theorem \ref{Thm.Charact:CS}, $Y$ is completely separated relative to $\XX$.
\end{example}

\begin{figure}[t]
  \centering
  \includegraphics[scale=0.14]{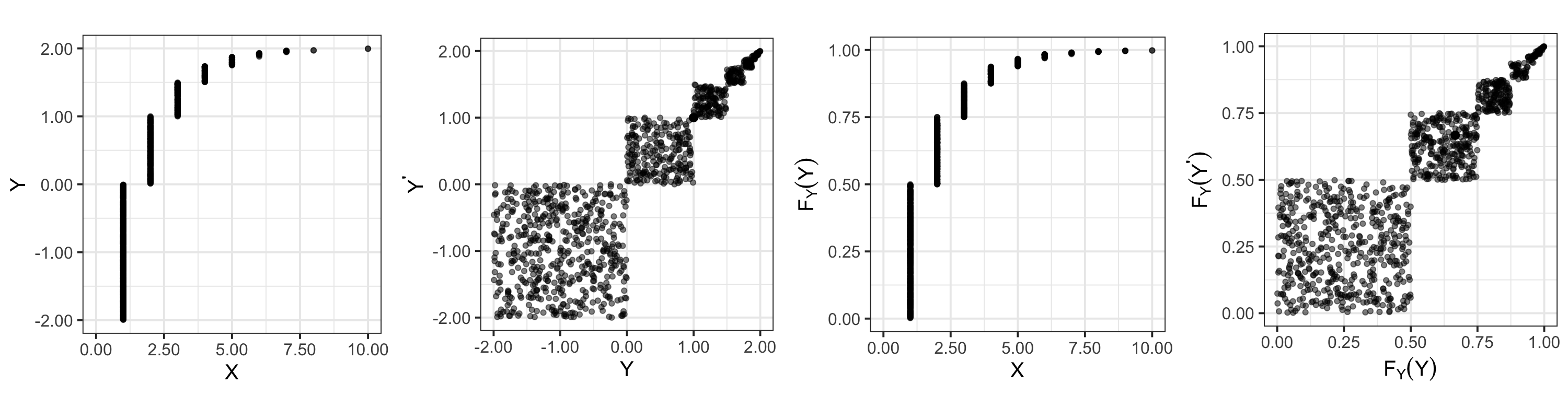}
  \caption{Scatterplots of sample size $1000$ showing $(X, Y)$, its Markov product $(Y,Y')$ and the transformation $(X, F_Y(Y))$ together with its corresponding Markov product $(F_Y(Y), F_Y(Y'))$ discussed in Example \ref{Kette}.}
  \label{E2P}
\end{figure}

\section*{Acknowledgments}

Both authors gratefully acknowledge the support of the Austrian Science Fund (FWF) project {P 36155-N} \emph{ReDim: Quantifying Dependence via Dimension Reduction}
and the support of the WISS 2025 project 'IDA-lab Salzburg' (20204-WISS/225/197-2019 and 20102-F1901166-KZP).


\end{document}